\documentclass[12pt, reqno,draft]{amsart}

\usepackage{amsmath}
\usepackage{amsthm}
\usepackage{amssymb}
\usepackage{amsrefs}
\usepackage{mathrsfs}
\usepackage[usenames]{color}

 \newtheorem{thm}{}[section]
 \newtheorem{theorem}[thm]{Theorem}
 \newtheorem{corollary}[thm]{Corollary}
 \newtheorem{lemma}[thm]{Lemma}
 \newtheorem{proposition}[thm]{Proposition}
  
  \theoremstyle{definition}
 \newtheorem{definition}[thm]{Definition}
 \theoremstyle{remark}
 \newtheorem{remark}[thm]{Remark}
 
 \newtheorem{question}[thm]{Question}
\newtheorem{example}[thm]{Example}
 
 \numberwithin{equation}{section}
 \allowdisplaybreaks

\newcommand{\LL}{\ensuremath{\mathcal{L}}}
\newcommand{\NN}{\ensuremath{\mathbb{N}}}

\newcommand{\xx}{\ensuremath{\mathbf{x}}}
\newcommand{\yy}{\ensuremath{\mathbf{y}}}
\newcommand{\zz}{\ensuremath{\mathbf{z}}}
\newcommand{\ee}{\ensuremath{\mathbf{e}}}
\newcommand{\ww}{\ensuremath{\mathbf{w}}}
\newcommand{\vv}{\ensuremath{\mathbf{v}}}
\newcommand{\gb}{\ensuremath{\mathbf{g}}}
\newcommand{\ff}{\ensuremath{\mathbf{f}}}
\newcommand{\dd}{\ensuremath{\mathbf{d}}}
\newcommand{\uu}{\ensuremath{\mathbf{u}}}
\newcommand{\WW}{\ensuremath{\mathcal{W}}}
\newcommand{\OO}{\ensuremath{\mathcal{O}}}

\newcommand{\FF}{\ensuremath{\mathbb{F}}}

\newcommand{\id}{\ensuremath{\mathrm{Id}}}

\newcommand{\RR}{\ensuremath{\mathbb{R}}}

\DeclareMathOperator{\supp}{supp}
\DeclareMathOperator{\rang}{R}

\begin{document}

\title[On Garling sequence spaces]{On Garling sequence spaces}

\author[F. Albiac]{Fernando Albiac}
\address{Mathematics Department\\ 
Universidad P\'ublica de Navarra\\
Campus de Arrosad\'{i}a\\
Pamplona\\ 
31006 Spain}
\email{fernando.albiac@unavarra.es}

\author[J. L. Ansorena]{Jos\'e L. Ansorena}
\address{Department of Mathematics and Computer Sciences\\
Universidad de La Rioja\\ 
Logro\~no\\
26004 Spain}
\email{joseluis.ansorena@unirioja.es}

\author[B. Wallis]{Ben Wallis}
\address{Department of Mathematical Sciences\\ Northern Illinois University\\ DeKalb, IL 60115 U.S.A.}
\email{benwallis@live.com}

\begin{abstract} The aim of this paper is to introduce and investigate a new class of separable Banach spaces modeled after an example of Garling from 1968.  For each $1\leqslant p<\infty$ and each  nonincreasing weight $\textbf{w}\in c_0\setminus\ell_1$ we exhibit an $\ell_p$-saturated, complementably homogeneous, and uniformly subprojective Banach space $g(\textbf{w},p)$.  We also show that $g(\textbf{w},p)$ admits a unique  subsymmetric basis  despite the fact that for a wide class of weights it does not admit a symmetric basis.  This provides the first known examples of Banach spaces where those two properties  coexist. \end{abstract}




\thanks{F. Albiac acknowledges the support of the  Spanish Ministry for Economy and Competitivity Grants MTM2014-53009-P  for \textit{An\'alisis Vectorial, Multilineal y Aplicaciones},  and MTM2016-76808-P for  \textit{ Operators, lattices, and structure of Banach spaces}.  J.L. Ansorena  acknowledges the support of the  Spanish Ministry for Economy and Competitivity Grant MTM2014-53009-P for  \textit{An\'alisis Vectorial, Multilineal y Aplicaciones}}

\maketitle

\section{Introduction, notation, and terminology}\label{Sec1}

\noindent Let  $1\le p<\infty$ and let $\ww=(w_n)_{n=1}^\infty$ be a weight, i.e., a sequence of positive scalars. 
Assume that $\ww$  is decreasing.
  Given a sequence  of scalars $f=(a_n)_{n=1}^\infty$ we put
\[
 \Vert  f  \Vert_{g(\ww,p)} = \sup_{\phi\in\OO } \left( \sum_{n=1}^\infty |a_{\phi(n)}|^p w_n \right)^{1/p}
 \]
where $\OO$ denotes the set of increasing functions from $\NN$ to $\NN$. 
We consider the normed space 
\[
g(\ww,p):=\{f=(a_n)_{n=1}^\infty\colon     \Vert  f  \Vert_{g(\ww,p)}<\infty\}.
\]
A straightforward application of Fatou's lemma gives that $g(\ww,p)$ is a Banach space. 
 Imposing
the further conditions $\ww\in c_0$ and $\ww\notin\ell_1$ will prevent us, respectively, from having $g(\ww,p)=\ell_p$ or $g(\ww,p)= \ell_\infty$. We will  assume as well  that $\ww$ is normalized, i.e., $w_1=1$.
Thus, we put
\[\WW:=\left\{(w_n)_{n=1}^\infty\in c_0\setminus\ell_1:1=w_1\geq w_2>\cdots w_n \ge w_{n+1} \ge\cdots>0\right\}.\]
The germ of this family of spaces goes back to \cite{Ga68}, where  Garling  showed that  for  $\ww=(n^{-1/2})_{n=1}^{\infty}$ the  canonical  unit vectors of $g(\ww, 1)$ form a subsymmetric basic sequence 
that is not  symmetric. Indeed, until then these two concepts were believed to be the same. 
Other Garling sequence spaces have appeared in past literature since then. In  \cite{Pu76}, Pujara  investigated  a variation of $g(\ww,2)$ for  $\ww=(n^{-1/2})_{n=1}^{\infty}$. He proved that those spaces are  uniformly convex, and that their canonical basis is subsymmetric but not symmetric. Dilworth el al. \cite{DOSZ11}  studied  derived forms of 
$g(\ww,1)$ for various choices of weights $\ww$ that yield   spaces in which the canonical basis  is 1-greedy and subsymmetric but not symmetric.

Of course, all these spaces are inspired in \textbf{sequence Lorentz spaces}  defined as
 \[
d(\ww,p)=\{f=(a_n)_{n=1}^\infty \colon \Vert  f\Vert_{d(\ww,p)}<\infty\},
\]
where, if $\Pi$ is the set of permutations on $\NN$, 
\[
 \Vert  f\Vert_{d(\ww,p)}=\sup_{\sigma\in\Pi}  \left(\sum_{n=1}^{\infty} |a_{\sigma(n)}|^p w_n\right)^{1/p}.
  \]
  The main difference between Lorentz sequence spaces and Garling sequence spaces  is that, unlike for $d(\ww, p)$, the canonical basis in $g(\ww, p)$ is not  symmetric  for a wide class of weights. We will tackle this issue in Section~\ref{sec4}. The canonical basis is  subsymmetric, though. Moreover,  while $d(\ww, p)$ is known to admit a unique symmetric basis, in Section~\ref{sec3} we prove that $g(\ww, p)$ has a unique subsymmetric basis.
  Along the way to proving the uniqueness of subsymmetric basis, we will
 show in  Section~\ref{sec2}  that many structural results of $d(\ww, p)$ carry over to $g(\ww, p)$.
To be precise, we will show that $g(\ww,p)$ is complementably homogeneous and uniformly complementably $\ell_p$-saturated.  As a consequence of the latter, $g(\ww,p)$ is uniformly subprojective.
Both subprojectivity and complementable homogeneity have been used in the study of the closed ideal structure of the algebra of endomorphisms of a Banach space.  Subprojectivity was introduced in \cite{Wh64} in order to study the closed operator ideal of strictly singular operators.  For instance, if $X$ and  $Y$ are Banach spaces, $Y$ is subprojective and $T\in\LL(X,Y)$, then  the dual operator $T^*$ is strictly singular only if $T$ is. If furthermore $X=Y$ and is reflexive, the converse holds, i.e., if $X$ is reflexive then $T\in\LL(X)$ is strictly singular if and only if $T^*$ is.  See \cite{OS15} for further discussion of subprojectivity.
Complementable homogeneity, meanwhile, has been used in several recent papers (cf., e.g., \cite{CJZ11,LSZ14,Zh14}) to show the existence of a unique maximal ideal in $\mathcal{L}(X)$.  In particular, the set of all $X$-strictly singular operators---those for which the restrictions to subspaces isomorphic to $X$ are never bounded below---acting on a complementably homogeneous Banach space $X$ forms the unique maximal ideal in $\mathcal{L}(X)$ whenever it is closed under addition.  By an argument in \cite[Corollary 5.2]{KPSTT12} and the second paragraph of the proof of \cite[Theorem 5.3]{KPSTT12}, it follows that the set of $g(\ww,p)$-strictly singular operators forms the unique maximal ideal in $\mathcal{L}(g(\ww,p))$.

We close this introductory section by  setting the notation and recalling the terminology that will be most heavily used. 
We will write  $\FF$ for the real or complex field and denote by
 $(\ee_n)_{n=1}^\infty$  the canonical basis of $\FF^\NN$, i.e.,  $\ee_n=(\delta_{k,n})_{k=1}^\infty$, were
$\delta_{k,n}=1$ if $n=k$ and $\delta_{k,n}=0$ otherwise. On occasion  we will need to compare linear combinations of the vectors $(\ee_n)_{n=1}^\infty$ with respect to different norms; thus to avoid  confusion we may  denote by $(\dd_n)_{n=1}^\infty$, $(\ff_n)_{n=1}^\infty$, and $(\gb_n)_{n=1}^\infty$  the  vectors $(\ee_{n})_{n=1}^{\infty}$ when seen, respectively,  inside the spaces  $d(\ww,p)$, $\ell_{p}$,  or $g(\ww,p)$. Also for simplicity, when $p$ and $\ww$ are clear from context the norms in the spaces $d(\ww,p)$ and $g(\ww,p)$ will be denoted by $\Vert\cdot\Vert_{d}$ and $\Vert\cdot\Vert_{g}$, respectively.
As it customary, we write $c_{00}$ for  the space of all scalar sequences with finitely many nonzero entries.  If $(\xx_n)_{n=1}^\infty$ is a basis for a Banach space $X$, and $f=\sum_{n=1}^\infty a_n \xx_n\in X$, then  the \textbf{support} of $f$ is the set $\supp(f):=\{n\in\NN:a_n\neq 0\}$ of indices corresponding to its nonzero entries. 
Given a basic sequence $(\zz_n)_{n=1}^\infty$ in $X$, $\langle \zz_n\rangle_{n=1}^{\infty}$ denotes its linear span and $[\zz_{n}]_{n=1}^{\infty}$ will be its closed linear span.

 Given families of positive real numbers $(\alpha_i)_{i\in I}$ and $(\beta_i)_{i\in I}$, the symbol $\alpha_i\lesssim \beta_i$ for $i\in I$ means that $\sup_{i\in I}\alpha_i/\beta_i <\infty$, while $\alpha_i\approx \beta_i$ for $i\in I$ means that $\alpha_i\lesssim \beta_i$ and $\beta_i\lesssim \alpha_i$ for $i\in I$. 
Now suppose $(\xx_n)_{n=1}^\infty$ and $(\yy_n)_{n=1}^\infty$ are basic sequences in $X$ and $Y$, respectively.  If for some positive $C$,
$$\left\Vert\sum_{n=1}^\infty a_n\xx_n\right\Vert_X\leq C\left\Vert\sum_{n=1}^\infty a_n\yy_n\right\Vert_Y,
\quad  (a_n)_{n=1}^\infty\in c_{00},$$
 we say that $(\yy_n)_{n=1}^\infty$ $C$-\textbf{dominates} $(\xx_n)_{n=1}^\infty$, and write $(\xx_n)_{n=1}^\infty\lesssim_C(\yy_n)_{n=1}^\infty$.  When the value of the constant is irrelevant we will simply say that $(\yy_n)_{n=1}^\infty$ \textbf{dominates} $(\xx_n)_{n=1}^\infty$, and write $(\xx_n)_{n=1}^\infty\lesssim(\yy_n)_{n=1}^\infty$.  Whenever $(\xx_n)_{n=1}^\infty\lesssim_C(\yy_n)_{n=1}^\infty$ and $(\yy_n)_{n=1}^\infty\lesssim_C(\xx_n)_{n=1}^\infty$, we say that $(\xx_n)_{n=1}^\infty$ and $(\yy_n)_{n=1}^\infty$ are $C$-\textbf{equivalent}, and write $(\xx_n)_{n=1}^\infty\approx_C(\yy_n)_{n=1}^\infty$ or, simply, $(\xx_n)_{n=1}^\infty\approx(\yy_n)_{n=1}^\infty$.  

By a \textbf{sign} we mean a  scalar of modulus one. 
A basic sequence $(\xx_n)_{n=1}^\infty$ is called \textbf{unconditional} if $(\xx_{\sigma(n)})_{n=1}^\infty$ is a basic sequence for any  $\sigma\in\Pi$. It is well known that a basic sequence  $(\xx_n)_{n=1}^\infty$  is unconditional if and only if there exists a constant $C\ge 1$ so that  $(\xx_n)_{n=1}^\infty\approx_C  (\epsilon_n\xx_n)_{n=1}^\infty$ for any choice of signs $(\epsilon_n)_{n=1}^\infty$. A basic sequence $(\xx_n)_{n=1}^\infty$ is called \textbf{subsymmetric} if it is unconditional and equivalent to all its subsequences.  It is called \textbf{symmetric} whenever it is  equivalent to each of its permutations.  If $(\xx_n)_{n=1}^\infty$ is a subsymmetric basic sequence then there is a  constant $C\geq 1$ such that $(\xx_n)_{n=1}^\infty\approx_C  (\epsilon_n \xx_{\phi(n)})_{n=1}^\infty$ for any $\phi\in\OO$ and any choice of signs $(\epsilon_n)_{n=1}^\infty$.  In this case we say that $(\xx_n)_{n=1}^\infty$ is $C$-\textbf{subsymmetric}.  Similarly, if $(\xx_n)_{n=1}^\infty$ is symmetric then there is $C\geq 1$ such that $(\xx_n)_{n=1}^\infty\approx_C (\epsilon_n \xx_{\sigma(n)})_{n=1}^\infty$ for any choice of signs $(\epsilon_n)_{n=1}^\infty$ and any $\sigma\in\Pi$, in which case we say that $(\xx_n)_{n=1}^\infty$ is $C$-\textbf{symmetric}.  Note that $C$-symmetry implies $C$-subsymmetry, which in turn implies $C$-unconditionality. Note also that every subsymmetric basis $(\xx_n)_{n=1}^\infty$ is \textbf{semi-normalized}, i.e., $\Vert \xx_n\Vert \approx 1$ for $n\in\NN$.

Given a function $\phi$ we denote 
by $\rang(\phi)$ its range.
Let  $\OO_f$ be the set of increasing functions from an integer interval $[1,\dots,r]\cap \NN$  into $\NN$.  Given 
 $\phi\in\OO_f$ we denote by $r(\phi)$ the largest integer in its domain.  
A function in $\OO_f$ is univocally determined by its range.
For $\ww=(w_n)_{n=1}^\infty$, $1\le p<\infty$, and $f=(a_n)_{n=1}^\infty\in\FF^\NN$ 
we have
\begin{equation}\label{SupGarling}
\Vert f\Vert_g^p=\sup_{\phi\in\OO_f}\sum_{n=1}^{r(\phi)} |a_{\phi(n)}|^p w_n.
\end{equation}

A Banach space with an unconditional (respectively, subsymmetric or symmetric) basis   is said to have a \textbf{unique unconditional} (respectively, \textbf{unique subsymmetric} or \textbf{unique symmetric basis}) if any two semi-normalized unconditional  (respectively, subsymmetric or symmetric) bases of $X$ are equivalent. 

Given  infinite-dimensional Banach spaces $X$ and $Y$, $\LL(X,Y)$ denotes the space of bounded linear operators from $X$ into $Y$, and  put $\LL(X)=\LL(X,X)$.
 The symbol $X\approx Y$ means that $X$ and $Y$ are isomorphic. 
The space $X$ is said to be $Y$-\textbf{saturated} (or  \textbf{hereditarily} $Y$) if every infinite-dimensional closed subspace of $X$ admits a further subspace $Z\approx Y$. If, in addition, there is a constant $C$ such that $Z$ can be chosen $C$-complemented in $X$ we say that $X$ is \textbf{uniformly complementably $Y$-saturated}. A Banach space $X$ is said to be \textbf{complementably homogeneous} whenever for every closed subspace $Y$ of $X$ such that $Y\approx X$, there exists another closed subspace $Z$ such that $Z$ is complemented in $X$, $Z\approx X$, and $Z\subseteq Y$. 
A Banach space $X$ is said to be \textbf{uniformly subprojective}  (see \cite{OS15}) if there is a constant $C$ such that for every infinite-dimensional subspace $Y\subseteq X$ there is a further subspace $Z\subseteq Y$ that is $C$-complemented in $X$.

Most other terminology is standard, such as might appear, for instance, in \cite{AK2016}.  

\section{Preliminaries}\label{prel}
\noindent
 In order to state some  embeddings involving Garling sequence spaces we will invoke
the \textit{weak Lorentz sequence space} $d_\infty(\ww,p)$ for $\ww\in\WW$ and $1\le <\infty$,  consisting of all  sequences $f=(a_{n})_{n=1}^{\infty}\in c_0$ so that 
  \[
\Vert f\Vert_{d_\infty(\ww,p)} =\sup_n \left(\sum_{k=1}^n w_k\right)^{1/p} a_n^*<\infty,
 \]
where
 $(a_n^*)_{n=1}^\infty$  denotes the decreasing rearrangement of $f$. Notice that $d_\infty(\ww,p)$ is a quasi-Banach space (see \cite{CRS2007}).

 \begin{lemma}\label{WeakLorentzChar}
 Suppose  $1\le p<\infty$ and $\ww=(w_n)_{n=1}^\infty\in\WW$. A sequence   $f=(a_n)_{n=1}^\infty\in \FF^{\NN}$ belongs to  the space $d_\infty(\ww,p)$ if and only if
 $D_{f}:=\sup_n \Vert \sum_{k=1}^n a_k \ee_k\Vert_{d_\infty(\ww,p)} <\infty$, in which case $\Vert f\Vert_{d_\infty(\ww,p)}=D_{f}$.
 \end{lemma}
 \begin{proof} We may  suppose  that $\{n\in\NN \colon a_n\not=0\}$ is infinite (otherwise there is nothing to prove). Put  $s_n= (\sum_{k=1}^n w_k)^{1/p}$. 
  Assume that $D_{f}<\infty$. In order to  show that $f\in c_0$ it suffices to see that for each $\varepsilon>0$ the set $\{k\in\NN \colon |a_k|\ge \varepsilon \}$ is finite.
 Let $m\in\NN$ be such that $D_{f}<\varepsilon  s_m$.  For a given $j$, denote by $(b_n^*)_{n=1}^\infty$ the decreasing rearrangement of 
  $\sum_{n=1}^j a_n \ee_n$. We have $b_m^*<\varepsilon$ and so $|\{ n\le j \colon |a_n|\ge\varepsilon \}|\le m-1$. Since $j$ is arbitrary, we are done.
  
  Now we consider the decreasing rearrangement $(a_n^*)_{n=1}^\infty$ of $f$. Given $n\in\NN$, pick $r\in\NN$ such that $|a_k|<a_n^*$ for $k>r$. We have $a_n^*=c_n^*$, where $(c_n^*)_{n=1}^\infty$ is the decreasing rearrangement of  $\sum_{n=1}^r a_n \ee_n$.  Therefore
  $a_n^* s_n \le D_{f}$. In other words, $f\in d_\infty(\ww,p)$ and $\Vert f\Vert_{d_\infty(\ww,p)}\le D_{f}$. 
  
  The reverse inequality  and the converse implication are obvious.
 \end{proof}

\begin{proposition}\label{embeddings} Let $1\le p<\infty$ and $\ww\in\WW$. Then 
\[\ell_p\subsetneq d(\ww,p)\subseteq g(\ww,p)\subseteq d_\infty(\ww,p),\] with norm-one inclusions.
\end{proposition}
\begin{proof}The embedding  $\ell_p\subsetneq d(\ww,p)$ is clear  and well known. 

Let $f=(a_n)_{n=1}^\infty\in \FF^\NN$ and $\phi\in\OO_f$. Pick $\sigma\in\Pi$ extending $\phi$.
Then 
\[
\sum_{n=1}^{r(\phi)} |a_{\phi(n)}|^p w_n=\sum_{n=1}^{r(\phi)} |a_{\sigma(n)}|^p w_n\le \sum_{n=1}^{\infty} |a_{\sigma(n)}|^p w_n\le \Vert f\Vert^p_{d(\ww,p)}.
\]
Taking the supremum on  $\phi$ we get $\Vert f\Vert_{g(\ww,p)} \le \Vert f\Vert_{d(\ww,p)}$.

Now let  $f=(a_n)_{n=1}^\infty\in c_{00}$. For a given  $n\in \NN$ there is $\phi\in\OO_f$ with $r(\phi)=n$ and
$| a_{\phi(k)}|\ge a_n^*$ for $1\le k \le r(\phi)=n$. We have
\[
a_n^*\left(  \sum_{k=1}^n  w_k\right)^{1/p}\le \left(  \sum_{k=1}^n    |a_{\phi(k)}|^p w_k\right)^{1/p} 
\le \Vert f\Vert_{d(\ww,p)},
\]
so that $\Vert f\Vert_{d_\infty(\ww,p)} \le \Vert f\Vert_{g(\ww,p)}$. An appeal to Lemma~\ref{WeakLorentzChar} concludes the proof.
\end{proof}

Before proceeding with the proof of the aforementioned subsymmetry of the canonical basis of Garling sequence spaces, let us introduce some linear maps that will be handy. 

$\bullet$ Given a sequence of signs $\epsilon=(\epsilon_n)_{n=1}^\infty$ we define $M_\epsilon\colon \FF^\NN \to \FF^\NN$ by
\[
 \quad  M_\epsilon((a_n)_{n=1}^\infty)=(\epsilon_n a_n)_{n=1}^\infty.
\]

$\bullet$  The coordinate projection on  $A\subseteq \NN$ is defined by
\[
P_A\colon \FF^\NN \to \FF^\NN, \quad  P_A((a_n)_{n=1}^\infty)=(\lambda_n a_n)_{n=1}^\infty,
\]
where $\lambda_n=1$ if $n\in A$ and $\lambda_n=0$ otherwise. \newline

$\bullet$ Given $\phi\in\OO$ we define $V_\phi\colon  \FF^\NN \to \FF^\NN$  by $
 V_\phi( (a_n)_{n=1}^\infty)=(a_{\phi(n)})_{n=1}^\infty$,
 
and

$\bullet$ $U_\phi\colon  \FF^\NN \to \FF^\NN$ by $U_\phi((a_n)_{n=1}^\infty)=(b_n)_{n=1}^\infty$, where
\[ b_n=\begin{cases} a_k & \text{ if }n=\phi(k), \\ 0 &\text{ if }n\notin \rang(\phi). \end{cases}
\]
\begin{remark}\label{rmk:1}Given $\phi\in\OO$ and $n\in\NN$,
We have $U_\phi \circ V_\phi=\id_{\FF^\NN}$ and $V_\phi(\ee_n)=\ee_{\phi(n)}$.
\end{remark}

\begin{proposition}\label{subsymmetry} Let $\ww\in \WW$ and $1\le p<\infty$. Let $\epsilon=(\epsilon_n)_{n=1}^\infty$ be a sequence of signs, let $A$ be  a subset of $\NN$, and  let $\phi$ a map in $\OO$.
\begin{itemize}
\item[(i)]  $M_\epsilon$ and $P_A$ are  norm-one operators from $g(\ww,p)$ into $g(\ww,p)$.
\item[(ii)]  $V_\phi$ and $U_\phi$ are norm-one operators from $g(\ww,p)$ into $g(\ww,p)$.
\item[(iii)] $V_\phi$ is an isometric embedding from $g(\ww,p)$ into $g(\ww,p)$.
\item[(iv)]  The standard unit vectors form a $1$-subsymmetric basic sequence in $g(\ww,p)$.
\end{itemize}
\end{proposition}

\begin{proof} (i) is clear; 
(iii) is a consequence of (ii) and Remark~\ref{rmk:1};   (iv) is a consequence of  (i),  (iii), and Remark~\ref{rmk:1}. Thus we must only  care to show (ii).
Let $f=(a_n)_{n=1}^\infty\in \FF^\NN$. It is obvious that  $\Vert M_\epsilon(f)\Vert_g=\Vert f\Vert_g$
and that $\Vert P_A(f)\Vert_g \le \Vert f\Vert_g$. Since $\phi\circ\psi\in\OO$ for all $\psi\in\OO$,
\[
\Vert V_\phi(f)\Vert_g^p=\sup_{\psi\in\OO} \sum_{n=1}^\infty |a_{\phi(\psi(n))}|^pw_n \le \Vert f\Vert_g.
\]
Let $U_\phi(f)=(b_n)_{n=1}^\infty$. Pick $\psi\in\OO$.
The function $\phi^{-1}\circ\psi$ is an increasing map from a set $A\subseteq \NN$ to $\NN$.  Put $J=\{ n\in\NN \colon n\le |A|\}$ and choose
 $\gamma\colon J \to A$ increasing and bijective. Since $\rho:=\phi^{-1}\circ\psi\circ\gamma\in\OO\cup \OO_f$ and $n\le \gamma(n)$ for all $n\in J$,
\begin{align*}
 \sum_{n=1}^{\infty} |b_{\psi(n)}|^p w_n
&=\sum_{n\in A}  |a_{\phi^{-1}\circ\psi(n)}|^p w_n
=\sum_{n\in J}  |a_{\rho(n)}|^p w_{\gamma(n)}\\
&\le \sum_{n\in J}  |a_{\rho(n)}|^p w_{n}
\le  \Vert f\Vert^p_{g}.
\end{align*}
Taking the supremum on $\psi$ we get $\Vert U_\phi(f)\Vert_g \le \Vert f\Vert_g$.
\end{proof}

\begin{definition}
Let $(\xx_n)_{n=1}^\infty$ be a basis  for a Banach space $X$ and let $g=\sum_{n=1}^\infty b_n \xx_n$ be a vector in $X$. Given another $f=\sum_{n=1}^\infty a_n \xx_n \in X$ 
 we write $f\prec g$  (with respect to the basis $(\xx_n)_{n=1}^\infty$) if  
\begin{itemize}
\item  $\supp(f)\subseteq\supp(g)$, 
 \item $a_n=b_n$ for all $n\in\supp(f)$, and
 \item   $\Vert f\Vert=\Vert g\Vert$. 
\end{itemize}
We say that $g$ is  \textbf{minimal}  in $X$ (with respect to   $(\xx_n)_{n=1}^\infty$)  if it is minimal in $X$ equipped with the partial order $\prec$, i.e., if
 $f\prec g$  implies $f=g$. 
\end{definition}

\begin{remark}
 Note that if $g$ is finitely supported  with respect to a basis $(\xx_n)_{n=1}^\infty$  then there exists a minimal $f\prec g$.
 \end{remark}
\begin{remark}  Every  sequence in $d(\ww,p)$, $1\le p<\infty$,  is minimal. The situation is quite different for $g(\ww,p)$.  For instance,  if $(1-w_2) t^p \ge 1$, then the vector $\gb_1+t \gb_2$ is not minimal.
\end{remark}

\begin{lemma}\label{ConstantCoefficients}  Let $\ww=(w_n)_{n=1}^\infty\in\WW$ and $1\le p<\infty$. For any $A\subseteq\NN$ finite we have:
\begin{itemize}
\item[(i)] $\Vert\sum_{n\in A} \ee_n\Vert_{g}=\sum_{n=1}^{|A|} w_n$.
\item[(ii)] The vector $\sum_{n\in A} \ee_n$ is minimal in $g(\ww,p)$ (with respect to its canonical basis).
\end{itemize}
\end{lemma}
\begin{proof}
Part~(i)  can be obtained, for example, from Proposition~\ref{embeddings}.
To prove (ii), let  $f\prec \sum_{n\in A} \ee_n$. Then $f=\sum_{n\in B} \ee_n$ for some $B\subseteq A$, and  
\[\sum_{n=1}^{|B|} w_n=\Vert f\Vert_g=\left\Vert\sum_{n\in A} \ee_n\right\Vert_{g}=\sum_{n=1}^{|A|} w_n.\]
 Consequently $|B|=|A|$ and so $B=A$.
\end{proof}

\begin{lemma}\label{NormAttaining}  Let $\ww\in\WW$ and $1\le p<\infty$. Suppose
$f\in c_{00}$ is minimal in $g(\ww,p)$ with respect to the canonical basis. Then 
\begin{equation}\label{attaining-function}
\Vert f\Vert_g=\left(\sum_{n=1}^{r(\psi)} |a_{\psi(n)}|^pw_n\right)^{1/p},
\end{equation}
where $\psi\in\OO_f$  is determined by  $\rang(\psi)=\supp f$.
\end{lemma}
\begin{proof}
Given $\phi\in\OO_f$, let $\psi$ be the map in $\OO_f$ determined by $\rang(\psi)=\rang(\phi)\cap\supp f$. 
Let $\gamma$ by the inverse function of $\phi$ restricted to $\rang(\psi)$.
We have
\[
\sum_{n=1}^{r(\phi)} |a_{\phi(n)}|^p w_n
=\sum_{n\in \rang(\psi)} |a_n|^p w_{\gamma(n)}\le
 \sum_{n\in \rang(\psi)} |a_n|^p w_{\psi^{-1}(n)}=\sum_{n=1}^{r(\psi)} |a_{\psi(n)}|^pw_n.
\]
Therefore in order to compute  the supremum in \eqref{SupGarling} we can restrict ourselves to $\{\phi\in\OO_f\colon \rang(\phi)\subseteq\supp f\}$, so that this supremum is attained.
Let $\psi$ be the element in $\OO_f$ with $\rang(\psi)\subseteq \supp(f)$ where the supremum  in \eqref{SupGarling}  is attained. 
Let  $h$ be the projection of $f$ onto $\rang(\psi)$. Then
\[
 \Vert f \Vert_g^p=\sum_{n=1}^{r(\psi)} |a_{\psi(n)}|^pw_n\le \Vert h\Vert^p_g,
 \]
 hence $h=f$ and $\rang(\psi)=\supp f$. \end{proof}

Recall that a \textbf{block basic sequence} of a basic sequence $(\xx_n)_{n=1}^\infty$ in a Banach space $X$ is a sequence  $(\yy_n)_{n=1}^\infty$ of non-zero vectors of the form
\[
\yy_n=\sum_{i=p_n}^{p_{n+1}-1} a_i \xx_i
\]
for some (unique) $(a_n)_{n=1}^\infty\in\FF^\NN$ and some  increasing sequence of integers $(p_n)_{n=1}^\infty$ with $p_1=1$.

\begin{definition} Let $\yy_n=\sum_{i=p_n}^{p_{n+1}-1} a_i \xx_i$, $ n\in \NN,$ be a block basic sequence of a basic sequence $(\xx_n)_{n=1}^\infty$ in a Banach space $X$. Then:
\begin{enumerate}
\item[(a)]  $(\yy_{n})_{n=1}^{\infty}$ is said to be  \textbf{uniformly null} if $\lim_n a_n=0$.
  \item[(b)]   $(\yy_{n})_{n=1}^{\infty}$ is said to verify the \textbf{gliding hump property} if  
 \begin{equation*}
 \inf_{n\in\NN}\sup_{p_n\leq i\leq p_{n+1}-1}|a_i|>0.
 \end{equation*}
 \end{enumerate}
 \end{definition}
 
  \begin{remark}\label{dichotomy-block} A subsequence of a block basic sequence  also is  a block basic sequence. Moreover, we have the following dichotomy: 
 a block basic sequence is either  uniformly null or has a subsequence verifying the gliding hump property.
 \end{remark}

 \begin{lemma}\label{largeblocks} Let  $(\xx_n)_{n=1}^\infty$ be a semi-normalized basis in a Banach space $X$. If  $(\yy_n)_{n=1}^\infty$ is a semi-normalized uniformly null block basic sequence of $(\xx_n)_{n=1}^\infty$ then
 $\lim_n |\supp(\yy_n)|=\infty$. 
 \end{lemma}
 \begin{proof} Assume the claim fails. Then there is an infinite subset $A\subseteq\NN$ with $D=\sup_{n\in A}|\supp(\yy_n)|<\infty$.
  Write $\yy_n=\sum_{i=p_n}^{p_{n+1}-1} a_i \xx_i$ and put 
  \[C_n=\max_{p_n\le i \le p_{n+1}-1}|a_i|,\quad
C=\sup_n \Vert \xx_n\Vert,\quad \text{and}\quad B=\inf_n \Vert y_n\Vert.\]
 We have $B \le CDC_n$ for all $n\in A$. Letting $n$ tend to infinity trough $A$
we obtain $B\le 0$, an absurdity.
 \end{proof}

We shall also need two properties on block basic sequences of subsymmetric bases that we gather in the following proposition. 

\begin{proposition}\label{subsymmetric-dominates}  Let $X$ be a Banach space with a subsymmetric basis $(\xx_n)_{n=1}^\infty$ and  suppose $(\yy_n)_{n=1}^\infty$ is a semi-normalized block basic sequence  of $(\xx_n)_{n=1}^\infty$.
\begin{enumerate}
 \item[(i)]  If $\sup_n|\supp (\yy_n)|<\infty$  then $(\yy_n)_{n=1}^\infty$ is equivalent to $(\xx_n)_{n=1}^\infty$.
\item[(ii)] If $(\yy_n)_{n=1}^\infty$ verifies the gliding hump property 
then $(\xx_n)_{n=1}^\infty\lesssim (\yy_n)_{n=1}^\infty$.
\end{enumerate}
\end{proposition}

\begin{proof}
The proofs  follow the steps of  \cite{ACL73}*{Proposition 3} and \cite{ACL73}*{Proposition 4}, respectively, and so we leave them out for the reader.
\end{proof}

To make this paper as self-contained as possible we  record the following well known fact.
\begin{proposition}[cf. \cite{ACL73}*{Proposition 1}]\label{subsymmetric-weakly-null} Every subsymmetric basic sequence in a Banach space is either weakly null or else  equivalent to the canonical basis of $\ell_1$.\end{proposition}

\begin{corollary}\label{garling-weakly-null}
Let $(\xx_n)_{n=1}^\infty$ be a basic sequence in a Banach space $X$. If $(\xx_n)_{n=1}^\infty$ is  equivalent  to the unit vector basis  of $g(\ww,p)$ for $\ww\in\WW$ and $1\le p<\infty$,  then it is weakly null.
\end{corollary}

\begin{proof} It is straightforward from Proposition~\ref{embeddings}, Proposition~\ref{subsymmetry}, and Proposition~\ref{subsymmetric-weakly-null}.
\end{proof}

 We put an end to this preliminary section recalling a precise version of Bessaga-Pelczy\'nski Selection Principle that we will need below.
 \begin{theorem}[cf. \cite{AK2016}*{Proposition 1.3.10}]\label{BPSP} Let $X$ be a Banach space with a basis $(\xx_n)_{n=1}^\infty$. Let $(\yy_n)_{n=1}^\infty$ be a 
 normalized weakly null sequence in $X$ and let $\epsilon>0$. Then there exists a subsequence  
 $(\yy_{n_k})_{n=1}^\infty$   of $(\yy_n)_{n=1}^\infty$ that is $(1+\epsilon)$-equivalent  to a normalized block basic sequence $(\zz_k)_{k=1}^\infty$ of $(\xx_n)_{n=1}^\infty$.
 Moreover, if $T\colon[\zz_k]_{k=1}^\infty \to  [\yy_{n_k}]_{n=1}^\infty$ is the isomorphism given by $T(\zz_k)=\yy_{n_k}$, we have that whenever $Z\subseteq [\zz_k]_{k=1}^\infty$ is 
 $C$-complemented in $X$ then $T(Z)$ is $C(1+\epsilon)$-complemented in $X$.
 \end{theorem}

\section{Geometric properties of Garling sequence spaces}\label{sec2}

\noindent Our ultimate goal in this section is to establish the following set of structural results about the spaces $g(\ww,p)$.
\begin{theorem}\label{main} Let $1\leq p<\infty$ and $\ww=(w_n)_{n=1}^\infty\in\WW$. \begin{itemize}
\item[(i)] The unit vectors in $g(\ww, p)$ span  the entire space  and  form a boundedly complete basis.
\item[(ii)]  Every basic sequence in $g(\ww,p)$ equivalent to its canonical basis admits a subsequence spanning a subspace complemented in $g(\ww,p)$.
\item[(iii)]  If $1<p<\infty$ then $g(\ww,p)$ is reflexive, but if $p=1$ then $g(\ww,1)$ is nonreflexive.
\item[(iv)]  No subspace of $\ell_p$ is isomorphic to $g(\ww,p)$.
\item[(v)]  For every $\epsilon>0$ and every infinite-dimensional closed subspace $Y$ of $g(\ww,p)$ there exists a further subspace $Z\subseteq Y$ that is $(1+\epsilon)$-isomorphic to $\ell_p$
and $(1+\epsilon)$-complemented in $g(w,p)$.

\item[(vi)] The identity operator on $\ell_p$ factors through the inclusion  map 
$I_{d,g}\colon d(\ww,p) \to g(\ww,p)$.

\end{itemize}\end{theorem}

To that end, let us begin with the following Proposition.

\begin{proposition}\label{lp-dominates}Let $1\leq p<\infty$ and $\ww=(w_n)_{n=1}^\infty\in\WW$.  Every normalized block basic sequence of the canonical basis of $g(\ww, p)$ is 1-dominated by  the canonical basis of $\ell_p$.\end{proposition}

\begin{proof} Let 
\[\yy_n=\sum_{i=p_n}^{p_{n+1}-1}a_i\gb_i,\quad n\in \mathbb N,\]
be a normalized block basic sequence. For  $i\in\NN$ define $n=n(i)\in\NN$ by $p_n\le i \le p_{n+1}-1$.
Let $(b_n)_{n=1}^\infty\in c_{00}$. Denote $(c_i)_{i=1}^\infty:=\sum_{n=1}^\infty b_n\yy_n$. We have $c_i=b_{n(i)} a_i$ for all $i\in\NN$.
Let $\phi\in\OO$.  For each $n\in\NN$,   $J_n:=\{i\in\NN \colon p_n\le \phi(i)\le p_{n+1}-1\}$ is an integer interval and, then, for some $q_n$, $r_n\in\NN\cup\{0\}$,
$J_n=\{j\in\NN \colon 1+q_n\le j\le q_n+r_n\}$.  We have
\begin{align*}
\sum_{n=1}^\infty |c_{\phi(n)}|^pw_n
&=\sum_{n=1}^\infty|b_n|^p  \sum_{i\in J_n}|a_{\phi(i)}|^pw_{i}\\
&=\sum_{n=1}^\infty|b_n|^p  \sum_{j=1}^{r_n} |a_{\phi(j+q_n)}|^p w_{j+q_n}\\
&\le\sum_{n=1}^\infty|b_n|^p  \sum_{j=1}^{r_n} |a_{\phi(j+q_n)}|^p w_{j}\\
&\le\sum_{n=1}^\infty|b_n|^p\, \Vert \yy_n\Vert^p_g\\
&=\sum_{n=1}^\infty|b_n|^p.
\end{align*}
Taking the supremum on $\phi$ yields 
$\left\Vert \sum_{n=1}^\infty b_n\yy_n\right\Vert^p \le \sum_{n=1}^\infty|b_n|^p.$
\end{proof}

 We are now able to state a powerful result. 
In its proof we will make use of the following construction.
  \begin{definition}\label{DefShiftLeft} Let  $(\xx_n)_{n=1}^\infty$ be a  basis in a Banach space $X$. Given a   block basic sequence  $(\yy_n)_{n=1}^\infty$ of  $(\xx_n)_{n=1}^\infty$, 
\[
\yy_n=\sum_{i=p_n}^{p_{n+1}-1} a_i \xx_i, \quad n\in\NN,
\]
 its \textbf{left-shifted block basic sequence}  $(\hat\yy_n)_{n=1}^\infty$ is the block basic sequence of $(\xx_n)_{n=1}^\infty$ constructed from $(\yy_n)_{n=1}^\infty$ as follows:
Consider $\phi\in\OO$ with range $\{ i\in\NN\colon a_i\not=0\}$ and put
\[
\hat\yy_n=\sum_{i=\hat p_n}^{\hat p_{n+1}-1}  a_{\phi(i)} \xx_i, \quad n\in\NN,
\]
where 
$\hat p_1=1$, and  $\hat p_{n+1}-\hat p_n=|\supp\yy_n|$ for all $n\in\NN$.
\end{definition}

\begin{theorem}\label{subsequence-equivalent-lp} Let $1\leq p<\infty$ and $\ww\in\WW$. Suppose $(\yy_n)_{n=1}^{\infty}$
is a uniformly null normalized block basic sequence of the canonical basis of $g(\ww, p)$.
Then for each $\epsilon>0$ there exists a subsequence $(\yy_{n_k})_{k=1}^\infty$ of $(\yy_n)_{n=1}^{\infty}$
  that is $(1+\epsilon)$-equivalent to the canonical basis  of $\ell_p$ and such that   $[{\yy}_{n_k}]_{k=1}^\infty$ is $(1+\epsilon)$-complemented in  $g(\ww, p)$.
  \end{theorem}
\begin{proof} By Proposition~\ref{subsymmetry}~(i), without loss of generality we can assume that each $\yy_n$ has nonnegative coefficients with respect to $(\gb_i)_{i=1}^\infty$.
For each $n\in\NN$ pick  a minimal $\tilde \yy_n\prec \yy_n$. Let 
\[
\hat\yy_n=\sum_{i=p_n}^{p_{n+1}-1}a_i\gb_i, \quad n\in\NN,
\]
be the left-shifted block basic sequence of $({\tilde\yy}_n)_{n=1}^\infty$,
where $(a_i)_{i=1}^\infty$ is a sequence of positive scalars converging to zero and  $(p_n)_{n=1}^\infty$  is an increasing sequence of integers with $p_1=1$.
 By Lemma~\ref{largeblocks} we have $\lim_n(p_{n+1}-p_n)=\infty$.
Let $\ww=(w_n)_{n=1}^\infty$ and pick $\alpha=(1+\epsilon)^{-p}\in(0,1)$. 
We claim that there exist  increasing sequences of positive integers $(q_k)_{k=1}^\infty$ and $(n_k)_{k=1}^\infty$ with $q_1=1$ so that  for each $k\in \NN$:\begin{itemize}
\item $q_{k+1}-q_k=p_{1+n_k}-p_{n_k}$,
\item $A_k:=\sum_{i=q_k}^{q_{k+1}-1} a_{i+p_{n_k}-q_k}^p w_{i}\ge \alpha$,
\end{itemize}

To prove the claim we proceed recursively.
Suppose that $q_k$ and $n_{k-1}$ have been constructed (for the case $k=1$ we  put $n_0=0$).
Since $\ww\in c_0$ we can find $L\in\NN$ so that
\[\sum_{n=L+1}^{L+q_k-1}w_n<\frac{1-\alpha}{2}.\]
Then we choose $M\in\NN$ such that for any $i\ge M$
\[
a_i<\left(\frac{1-\alpha}{2L}\right)^{1/p}.
\]
 Select $j=n_k\in\NN$ so that $j>n_{k-1}$, $p_j\ge M$ and $p_{j+1}-p_j>L+q_k$. Set 
  $q_{k+1}=p_{j+1}-p_j+q_k$.  Let $\phi\in\OO$ determined by 
\[
\rang(\phi)=A:=\cup_{n=1}^\infty \supp(\tilde \yy_n).
\]
We have $\tilde\yy_j=V_\phi(\hat\yy_j)$. Hence, by
 Lemma~\ref{subsymmetry}~(iii), $\hat \yy_j$ is minimal. Then, by Lemma~\ref{NormAttaining},
\[1=\Vert \hat\yy_j\Vert_{g}^p=\sum_{n=1}^{p_{j+1}-p_j} a_{p_j+n-1}^pw_{n}
=\sum_{i=q_k}^{q_{k+1}-1} a_{p_j-q_k+i}^p w_{i-q_k+1}
,\]
and so
\begin{align*}
1-\sum_{i=q_k}^{q_{k+1}-1} a_{i+p_j-q_k}^p w_{i}
&=\sum_{i=q_k}^{q_{k+1}-1}  a_{i+p_j-q_k}^p (w_{i-q_k+1} -w_{i})\\
&\le \sum_{i=q_k}^{q_{k+1}-1} \frac{1-\alpha}{2L} (w_{i-q_k+1} -w_{i})\\
&\le  \sum_{i=q_k}^{q_k+L-1}  \frac{1-\alpha}{2L} +  \sum_{i=q_k+L}^{q_{k+1}-1} (w_{i-q_k+1} -w_{i})\\\
&=   \frac{1-\alpha}{2} +  \sum_{n=L+1}^{q_{k+1}-q_k}  w_n - \sum_{n=L+q_k}^{q_{k+1}-1} w_{n}\\
&=   \frac{1-\alpha}{2} +  \sum_{n=L+1}^{L+q_k-1}  w_n - \sum_{n=q_{k+1}-q_k+1}^{q_{k+1}-1} w_{n}\\
&\le  \frac{1-\alpha}{2} + \frac{1-\alpha}{2} =1-\alpha.
\end{align*}
This completes the proof of our claim. 
 
 Consider the linear map $S\colon\FF^\NN\to\FF^\NN$ given by
\[
S((c_i)_{i=1}^\infty)
=\left(\frac{1}{A_k} \sum_{i=q_k}^{q_{k+1}-1}  a_{i+p_{n_k}-q_{k}}^{p-1}w_i c_i\right)_{k=1}^\infty.
\]
 Let $f=(c_i)_{i=1}^\infty \in \FF^\NN$. H\"older's inequality gives
 \begin{align*}\left\Vert S(f)\right\Vert_p^p
&=\sum_{k=1}^\infty \frac{1}{A_k^p} \left|\sum_{i=q_k}^{q_{k+1}-1} c_i a_{i+p_{n_k}-q_{k}}^{p-1}w_i \right|^p\\
&\leq \sum_{k=1}^\infty \frac{1}{A_k^p} \left(\sum_{i=q_k}^{q_{k+1}-1}|c_i|^pw_i\right)\left(\sum_{i=q_k}^{q_{k+1}-1}a_{i+p_{n_k}-q_{k}}^pw_i\right)^{p-1}\\
&=\sum_{k=1}^\infty \frac{1}{A_k} \left(\sum_{i=q_k}^{q_{k+1}-1}|c_i|^pw_i\right)\\ 
&\leq\alpha^{-1}
\sum_{k=1}^\infty  \left(\sum_{i=q_k}^{q_{k+1}-1}|c_i|^pw_i\right)\\
&=\alpha^{-1}
\Vert f\Vert_{g}^p.\end{align*}
That is, $S$ is an $\alpha^{-1/p}$-bounded  operator  from  $g(\ww,p)$   into $\ell_p$.

The left-shifted block basic sequence of $(\hat\yy_{n_k})_{k=1}^\infty$,   which we denote 
by  $(\zz_k)_{k=1}^\infty$, is given by the formula
\[
\zz_k=\sum_{i=q_k}^{q_{k+1}-1} a_{i+p_{n_k}-q_k} \gb_i,   \quad  k\in\NN.
\] 
Since the left-shifted block basic sequences of $(\tilde\yy_{n_k})_{n=1}^\infty$ and 
$(\hat{\yy}_{n_k})_{k=1}^\infty$ coincide, there is $\psi\in\OO$ such that $V_\psi(\zz_k)=\tilde\yy_{n_k}$ for all $k\in \NN$. Let $T=S\circ U_\psi \circ P_A$.
 By Proposition~\ref{subsymmetry}~(i) and (ii), $T$ is a $(1+\epsilon)$-bounded operator 
from
$g(\ww,p)$ into  $\ell_p$.

 By Proposition~\ref{lp-dominates},  $ (\yy_{n_k})_{k=1}^{\infty} \lesssim_1 (\ff_k)_{k=1}^\infty$.  In other words, there is a  norm-one operator  
 $R\colon \ell_p \to g(\ww,p)$ with $R(\ff_k)=\yy_{n_k}$ for all $k\in\NN$.
For every $k\in\NN$ we have 
 \begin{align*}
(T\circ R)(\ff_k)&= S(U_\psi(P_A( R(\ff_k))))\\
&=S(U_\psi(P_A(\yy_{n_k})))\\
 &=S(U_\psi(\tilde\yy_{n_k}))\\
 &=S(U_\psi(V_\psi(\zz_k)))\\
 &=S(\zz_k)\\
 &=\frac{1}{A_k}\left( \sum_{i=q_k}^{q_{k+1}-1}a_{i+p_{n_k}-q_k} a_{i+p_{n_k}-q_k}^{p-1}w_i\right) \ff_k\\
 &=\ff_k.
 \end{align*}
Consequently, $T\circ R=\id_{\ell_p}$. We infer that $R\circ T$ is a projection from 
$g(\ww,p)$ onto $[\yy_{n_k}]_{k=1}^\infty$ and that 
$ (\ff_k)_{k=1}^\infty\lesssim_{1+\epsilon}(\yy_{n_k})_{k=1}^\infty$.
\end{proof}

\begin{remark}If we consider the natural lattice structure  both in $g(\ww,p)$ and in $\ell_p$,
the maps $R$ and $T$ constructed in the proof of Theorem~\ref{subsequence-equivalent-lp} are  lattice homomorphisms.

\end{remark}

\begin{corollary}\label{dominates-or-equivalent} Every normalized block basic sequence of $(\gb_n)_{n=1}^\infty$ admits a subsequence that dominates $(\gb_n)_{n=1}^\infty$.\end{corollary}

\begin{proof}This is an immediate consequence of Remark~\ref{dichotomy-block}, Proposition~\ref{subsymmetric-dominates}~(ii), Theorem \ref{subsequence-equivalent-lp}, 
and the the embedding $\ell_p\subseteq g(\ww,p)$ provided by Proposition~\ref{embeddings}, which reads as $(\gb_n)_{n=1}^\infty\lesssim (\ff_n)_{n=1}^\infty$.
\end{proof}

Before proving Theorem~\ref{main} we need a couple more results.

\begin{proposition}\label{boundedly-complete} Let $1\leq p<\infty$ and $\ww\in\WW$. 
 If  $(\yy_n)_{n=1}^\infty$ is a semi-normalized block basic sequence of  the canonical basis of $g(\ww, p)$ then
\[\lim_{m\to \infty}\left\Vert\sum_{n=1}^m\yy_n\right\Vert_{g}=\infty.\]
\end{proposition}

\begin{proof} 
From Remark~\ref{dichotomy-block} in combination with Proposition~\ref{subsymmetric-dominates}~(ii) and Theorem~\ref{subsequence-equivalent-lp} we infer that we can find a subsequence $(\yy_{n_k})_{k=1}^\infty$ that either dominates $(\gb_n)_{n=1}^\infty$ or else is equivalent to $(\ff_n)_{n=1}^\infty$. 
 In either case, since, by Proposition~\ref{embeddings}, $(\ff_n)_{n=1}^\infty$ dominates $(\gb_n)_{n=1}^\infty$, we have $(\gb_n)_{n=1}^\infty\lesssim (\yy_{n_k})_{k=1}^\infty$. 
 Put $\ww=(w_n)_{n=1}^\infty$. Appealing to the $1$-unconditionality of $(\gb_n)_{n=1}^\infty$
 and to Lemma~\ref{ConstantCoefficients}, 
 \[
\sum_{k=1}^m w_k= \left\Vert\sum_{k=1}^m\gb_k\right\Vert_{g}
 \lesssim \left\Vert\sum_{k=1}^m\yy_{n_k}\right\Vert_{g}
 \le \left\Vert\sum_{n=1}^j\yy_n\right\Vert_{g}, \quad m\in\NN, \, j\ge n_m.
 \]
 Since $\ww\notin\ell_1$ we are done.
\end{proof}

\begin{proposition}\label{block-basis-equivalent-lp} Let $1\leq p<\infty$ and $\ww=(w_n)_{n=1}^\infty\in\WW$.  Suppose $(\yy_n)_{n=1}^\infty$ is a block basic sequence of the canonical basis of $g(\ww,p)$. Then for any $\epsilon>0$ there exists a further block basic sequence of $(\yy_n)_{n=1}^\infty$ that is $(1+\epsilon)$-equivalent to the canonical basis of $\ell_p$ and $(1+\epsilon)$-complemented in $g(\ww,p)$.
\end{proposition}

\begin{proof}  Without loss of generality we assume
 that $(\yy_n)_{n=1}^\infty$ is  normalized.  By Proposition~\ref{boundedly-complete} we can  recursively construct an increasing sequence $(p_n)_{n=1}^\infty$ of positive integers such that
$ \left\Vert\sum_{i=p_n}^{p_{n+1}-1}\yy_i\right\Vert_{g} \ge n$ for all $n\in\NN$.
Then the block basic sequence
\[\zz_n=\frac{\sum_{i=p_n}^{p_{n+1}-1}\yy_i}{\Vert\sum_{i=p_n}^{p_{n+1}-1}\yy_i\Vert_{g}},\quad n\in\NN,\]
   is  uniformly null,  and so by Theorem~\ref{subsequence-equivalent-lp} we can find a subsequence $(\zz_{n_k})_{k=1}^\infty$ that is $(1+\epsilon)$-equivalent to the canonical basis of $\ell_p$
and $(1+\epsilon)$-complemented in $g(\ww,p)$.\end{proof}

\begin{proof}[Completion of the Proof of Theorem~\ref{main}]\
\smallskip

(i)  
By Proposition~\ref{boundedly-complete}, no block basic sequence of $(\gb_n)_{n=1}^\infty$ is equivalent to the canonical basis of $c_0$. Appealing to  \cite[Theorem 3.3.2]{AK2016}) we infer that 
$(\gb_n)_{n=1}^\infty$ is a boundedly complete basic sequence. 

Let us prove that $[\gb_n]_{n=1}^\infty=g(\ww,p)$.
Take $f=(a_n)_{n=1}^\infty\in g(\ww,p)$. Then $\sup_m \Vert \sum_{n=1}^m a_n \gb_n\Vert_g<\infty.$
Since $(\gb_n)_{n=1}^\infty$ is boundedly complete, the series
$\sum_{n=1}^\infty a_n \gb_n$ converges to some $h\in g(\ww,p)$. Obviously, $f=h$.

\noindent

(ii)   Let  $(\yy_n)_{n=1}^\infty$ be a basic sequence in $g(\ww,p)$ equivalent to its canonical basis.
 By Corollary~\ref{garling-weakly-null},  $(\yy_n)_{n=1}^\infty$ is weakly null. Then,  applying 
 Theorem~\ref{BPSP} 
  and passing to  a subsequence, we can assume that $(\yy_n)_{n=1}^\infty$ is a block basic sequence of $(\gb_n)_{n=1}^\infty$.
By Theorem \ref{subsequence-equivalent-lp} and Remark~\ref{dichotomy-block}, since $(\gb_n)_{n=1}^\infty$ has no subsequence equivalent to the $\ell_p$-basis, passing to a further subsequence we have that
$(\yy_n)_{n=1}^\infty$ has the gliding hump property. Let $K\ge 1$ be such that
$(\yy_n)_{n=1}^\infty\lesssim_K (\gb_n)_{n=1}^\infty$.  
Write 
\[\yy_n=\sum_{i=p_n}^{p_{n+1}-1}a_i \gb _i, \quad n\in\NN,\]
 for some $(a_n)_{n=1}^\infty\in\FF^\NN$ and  some increasing  sequence $(p_n)_{n=1}^\infty$ with $p_1=1$.
 For each $n\in\NN$ pick an integer sequence $({i_n})_{n=1}^\infty$ with $p_n\leq i_n<p_{n+1}$ and so that $\inf_{n\in\NN}|a_{i_n}|:=c>0$.  Observe that for any $(b_n)_{n=1}^\infty\in c_{00}$, by 1-subsymmetry of $(\gb_n)_{n=1}^\infty$ we have
\[
\left\Vert\sum_{n=1}^\infty \frac{b_{i_n}}{a_{i_n}} \yy_n\right\Vert_g
\leq K\left\Vert\sum_{n=1}^\infty\frac{b_{i_n}}{a_{i_n}}\gb_{i_n}\right\Vert_g
\leq\frac{K}{c}\left\Vert\sum_{n=1}^\infty b_{i_n}\gb_{i_n}\right\Vert_g
\leq\frac{K}{c}\left\Vert\sum_{n=1}^\infty b_n\gb_n\right\Vert_g.
\]
Thus we can define a bounded linear map $P\colon g(\ww,p)  \to [\yy_n]_{n=1}^\infty$ by the rule
\[\sum_{n=1}^\infty b_n\ee_n\mapsto \sum_{n=1}^\infty \frac{b_{i_n}}{a_{i_n}}\yy_n.\]
Since $P(\yy_n)=\yy_n$ for all $n\in\NN$ we are done.

 (iii)    Let us suppose that $(\gb_n)_{n=1}^\infty$ fails to be shrinking.  Then due to unconditionality of $(\gb_n)_{n=1}^\infty$, there exists a block basic sequence $(\yy_n)_{n=1}^\infty$ of $(\gb_n)_{n=1}^\infty$ that is equivalent to the canonical $\ell_1$-basis (see, e.g., \cite[Theorem 3.3.1]{AK2016}).  Now, by Proposition \ref{block-basis-equivalent-lp}, we can find a further block basic sequence equivalent to the canonical $\ell_p$-basis.  However, this is impossible if $p\neq 1$.  Thus, for $p>1$, $(\gb_n)_{n=1}^\infty$ is shrinking, and since it is also boundedly complete by part~(i), and unconditional, it must span a reflexive space by a theorem of James (see, e.g., \cite[Theorem 3.2.19]{AK2016}).  This proves the first part of (iii). The second part about nonreflexivity in case $p=1$ follows since $g(\ww,1)$ contains a (nonreflexive) copy of $\ell_1$ by Proposition~\ref{block-basis-equivalent-lp}.

 (iv)  Assume that $g(\ww,p)$ embeds into $\ell_p$. Then $\ell_p$ contains a basic sequence
$(\yy_n)_{n=1}^\infty\approx (\gb_n)_{n=1}^\infty$.
By Corollary~\ref{garling-weakly-null}, $(\yy_n)_{n=1}^\infty$ is 
 weakly null. Hence (see, e.g., \cite[Proposition 2.1.3]{AK2016})  $(\yy_n)_{n=1}^\infty$ has a sequence equivalent to the canonical basis $(\ff_n)_{n=1}^\infty$ of $\ell_p$.
 By Proposition~\ref{subsymmetry},  $(\gb_n)_{n=1}^\infty\approx (\ff_n)_{n=1}^\infty$, which it  false by Proposition~\ref{embeddings}.
 
 (v)  Fix $\epsilon>0$, and let $Y$ be any infinite-dimensional closed subspace of $g(\ww,p)$.  By a classical result of Mazur (see, e.g., \cite[Theorem 1.4.5]{AK2016}),  we can find a normalized basic sequence $(\yy_n)_{n=1}^\infty$ in $Y$.  In case $p>1$, this basic sequence must be weakly null by the reflexivity of $g(\ww,p)$.  For the case $p=1$, we can assume by Rosenthal's $\ell_1$ Theorem (see, e.g., \cite[Theorem 11.2.1]{AK2016}) together with James' $\ell_1$ Distortion Theorem (see, e.g., \cite[Theorem 11.3.1]{AK2016}) that $(\yy_n)_{n=1}^\infty$ admits a weakly Cauchy subsequence.  By passing to a further subsequence if necessary, the sequence
\[\left(\frac{\yy_{2n+1}-\yy_{2n}}{\Vert\yy_{2n+1}-\yy_{2n}\Vert_g}\right)_{n=1}^\infty\]
is normalized and weakly null.  Let us relabel if necessary so that in all cases $1\leq p<\infty$ we have found a normalized and weakly null basic sequence $(\yy_n)_{n=1}^\infty$ in $Y$.  By Theorem~\ref{BPSP} we can pass to a subsequence if necessary so that $(\yy_n)_{n=1}^\infty$ is $(\sqrt{1+\epsilon})$-equivalent to a normalized block basic sequence $(\yy'_n)_{n=1}^\infty$ of the canonical basis of $g(\ww,p)$. Moreover, if $Z\subseteq [\yy_n']_{n=1}^\infty$ is $C$-complemented in $g(\ww,p)$ and $T$  is the operator defined by
$T(\yy'_n)=\yy_n$,
 then $T(Z)$ is $(\sqrt{1+\epsilon})C$-complemented in 
 $g(\ww,p)$.
  Now we apply Proposition \ref{block-basis-equivalent-lp} to find  a block basic sequence $(\zz'_n)_{n=1}^\infty$ of $(\yy'_n)_{n=1}^\infty$ which is $(\sqrt{1+\epsilon})$-equivalent to the canonical basis $(\ff_n)_{n=1}^\infty$ of $\ell_p$ 
  such that  $[\zz'_n]_{n=1}^\infty$ is $(\sqrt{1+\epsilon})$-complemented in $g(\ww,p)$.  For $n\in\NN$ let
 $\zz_n=T(\zz'_n)$
It is clear that $(\zz_n)_{n=1}^\infty$ is $(\sqrt{1+\epsilon})$-equivalent to $(\zz'_n)_{n=1}^\infty$, and hence $(1+\epsilon)$-equivalent to $(\ff_n)_{n=1}^\infty$.
It follows that $[\zz_n]_{n=1}^\infty$ is a subspace of $Y$ that is $(1+\epsilon)$-isomorphic to $\ell_p$ and 
$(1+\epsilon)$-complemented in $g(\ww,p)$.

 (vi) 
The block basic sequence
$$\yy_n=\frac{\sum_{i=2^{n-1}}^{2^n-1}\gb_i}{\sum_{i=1}^{2^{n-1}} w_i},\quad n\in\NN.$$
is normalized and uniformly null. By Lemma~\ref{ConstantCoefficients}, each $\yy_n$ is minimal. Thus, given $\epsilon>0$, by Theorem~\ref{subsequence-equivalent-lp} we can find a subsequence $(\yy_{n_k})_{k=1}^\infty$ of $(\yy_{n})_{n=1}^{\infty}$ that is $(1+\epsilon)$-equivalent to the canonical $\ell_p$-basis and such that $[\yy_{n_k}]_{k=1}^\infty$ is  $(1+\epsilon)$-complemented in $g(\ww,p)$. Let
$P\colon g(\ww,p)\to [\yy_{n_k}]_{k=1}^\infty$ be a projection with $\Vert P \Vert\le 1+\epsilon$ and  let $S\colon  [\yy_{n_k}]_{k=1}^\infty\to \ell_p$ be given by $S(\yy_{n_k})=\ff_k$. We have $\Vert S \Vert\le 1+\epsilon$.
Pick $(\zz_n)_{n=1}^\infty$ in $d(\ww,p)$ so that
  $\yy_n=I_{d,g}(\zz_n)$. The sequence $(\zz_n)_{n=1}^\infty$
 is also normalized. Then, by the  $d(\ww,p)$-analog of  Proposition~\ref{lp-dominates} (see \cite{ACL73}*{Proposition 5}) $(\zz_{n_k})_{n=1}^\infty$  is $1$-dominated by  the canonical basis of $\ell_p$, that is,  the linear operator 
$T\in \LL( \ell_p, d(\ww,p))$  given by $T(\ff_k)=\zz_{n_k}$ for $k\in\NN$  verifies  $\Vert T\Vert \le 1$. The composition operator $S\circ P \circ I_{d,g} \circ T$ is the identity map  on $\ell_p$.
\end{proof}

The  longed for features of Garling sequence spaces that were advertised in Section~\ref{Sec1}  follow now readily  from Theorem~\ref{main}.
\begin{corollary} Let $\ww\in\WW$ and $1\le p<\infty$. The space 
$g(\ww,p)$ is  complementably homogeneous and  uniformly complementably $\ell_p$-saturated.
\end{corollary}

\begin{remark} There are $d(\ww,p)$-analogs of Theorem \ref{main}~(ii)  and  Theorem \ref{main}~(v) (cf.   \cite{CL74}*{Corollary 12} and \cite{ACL73}*{Theorem 1}, respectively).
 Then, for $\ww\in\WW$ and $1\le p<\infty$,   $d(\ww,p)$  is also
complementably homogeneous and  uniformly complementably $\ell_p$-saturated.
\end{remark}

\section{Uniqueness of subsymmetric basis in $g(\ww,p)$}\label{sec3}

\noindent  The spaces $g(\ww,p)$ do not have a unique unconditional basis since they are not isomorphic to any of the only three spaces that enjoy that property, namely  $c_{0}$, $\ell_{1}$, or $\ell_{2}$ (see, e.g., \cite{AK2016}*{Theorem  9.3.1}). The aim of this section is to show that they  do have a unique subsymmetric basis.

\begin{theorem}\label{unique-subsymmetric} Let $1\leq p<\infty$ and $\ww \in\WW$.  Suppose $(\yy_n)_{n=1}^\infty$ is a subsymmetric basic sequence in $g(\ww,p)$.  Then every subsymmetric basis for   $Y=[\yy_n]_{n=1}^\infty$ is equivalent to $(\yy_n)_{n=1}^\infty$.  In particular, every subsymmetric basis in $g(\ww,p)$ is equivalent to its canonical basis.\end{theorem}

\begin{proof}Assume $Y$ is not isomorphic to $\ell_p$, else we are done.  Now let $(\uu_n)_{n=1}^\infty$ be any subsymmetric basis for $Y$. Since
neither $(\yy_n)_{n=1}^\infty$ nor $(\uu_n)_{n=1}^\infty$ are equivalent to the canonical basis of $\ell_1$, 
  Proposition~\ref{subsymmetric-weakly-null}  gives that both  sequences  are weakly null.  Thus, we may apply 
 Theorem~\ref{BPSP} 
   to obtain  $(\uu'_n)_{n=1}^\infty$ a normalized block basic sequence of $(\yy_n)_{n=1}^\infty$ and $(\yy'_n)_{n=1}^\infty$ 
  a normalized block basic sequence of $(\uu_n)_{n=1}^\infty$ with $(\uu'_n)_{n=1}^\infty\approx(\uu_n)_{n=1}^\infty$ and $(\yy'_n)_{n=1}^\infty\approx(\yy_n)_{n=1}^\infty$. 
By Theorem \ref{subsequence-equivalent-lp} neither $(\uu'_n)_{n=1}^\infty$ nor $(\yy'_n)_{n=1}^\infty$ are uniformly null since,
otherwise, either $(\uu_n)_{n=1}^\infty$ or $(\yy_n)_{n=1}^\infty$ would be equivalent to the canonical basis of $\ell_p$.  Combining  subsymmetry
with Remark~\ref{dichotomy-block} it follows that, passing to subsequences if necessary, we may assume that both $(\uu'_n)_{n=1}^\infty$ and $(\yy'_n)_{n=1}^\infty$  verify the gliding hump property. Proposition~\ref{subsymmetric-dominates}~(ii) yields
$$
(\yy_n)_{n=1}^\infty\lesssim(\uu'_n)_{n=1}^\infty\approx(\uu_n)_{n=1}^\infty\lesssim(\yy'_n)_{n=1}^\infty\approx(\yy_n)_{n=1}^\infty.
$$\end{proof}



\begin{remark} It is also easy to check that Theorem~\ref{unique-subsymmetric}  holds as well  for $d(\ww,p)$ in place of $g(\ww,p)$ and so
$d(\ww,p)$ admits a unique subsymmetric basis (which is also symmetric).  \end{remark}

\section{Nonexistence of symmetric basis in $g(\ww, p)$}\label{sec4}

\noindent In 2004, B. Sari  \cite[\S6]{Sa04} proved that there exist Banach spaces, called Tirilman spaces, admitting a subymmetric basis but failing to admit a symmetric one  (in fact, those spaces do not even contain symmetric basic sequences!). The aim of this section is, on the one hand, to determine when $g(\ww,p)$ admits a symmetric basis, and on the other hand, to show that for a wide class  of weights $\ww\in\WW$, the space $g(\ww,p)$ does not have such a basis despite having a unique subsymmetric basis.



\begin{theorem}\label{tfae} Let $1\leq p<\infty$ and $\ww=(w_n)_{n=1}^\infty\in\WW$.
 The following are equivalent.
\begin{enumerate}\item[(i)]  The canonical basis  is a symmetric basis in $g(\ww,p)$.
\item[(ii)] $g(\ww,p)$ admits a symmetric basis.
\item[(iii)]  The canonical bases of  $g(\ww, p)$ and $d(\ww,p)$ are equivalent, that is $g(\ww.p)=d(\ww,p)$ up to an equivalent norm.
\item[(iv)]  $g(\ww,p)$ and $d(\ww,p)$ are isomorphic.
\item[(v)]  $d(\ww,p)$ contains a subspace isomorphic to $g(\ww,p)$.
\item[(vi)] The inclusion map 
 $I_{d,g}\colon  d(\ww,p) \to g(\ww,p)$ 
 preserves a copy of $d(\ww,p)$ that is complemented in $g(\ww,p)$.  More precisely, there exists a subspace $Y$ of $d(\ww,p)$ such that $Y\approx d(\ww,p)$,  $I_{d,g}|_Y$ is bounded below, and  $I_{d,g}(Y)$ is complemented in $g(\ww,p)$.
\item[(vii)]  There is an operator  $T\in \LL(d(\ww,p),g(\ww,p))$
such that $T\circ I_{d,g}$  preserves a copy of $d(\ww,p)$.
\end{enumerate}\end{theorem}

\begin{proof} The implications (i) $\Rightarrow$ (ii),  (iii) $\Rightarrow$ (i), (iii) $\Rightarrow$ (iv), (iv) $\Rightarrow$ (v), and (iii) $\Rightarrow$ (vi) are all trivial.

(ii) $\Rightarrow$ (i) A symmetric basis  for $g(\ww,p)$  is also subsymmetric and so, by Theorem~\ref{unique-subsymmetric},
equivalent to $(\gb_n)_{n=1}^\infty$. Hence $(\gb_n)_{n=1}^\infty$ is symmetric.

(i) $\Rightarrow$ (iii) There is uniform constant $K$ such that $(\gb_{\sigma(n)})_{n=1}^\infty\approx_K
(\gb_n)_{n=1}^\infty$ for all $\sigma\in\Pi$.
Given $f=(a_n)_{n=1}^\infty\in c_0$ and $\sigma\in\Pi$ we have
\[
\sum_{n=1}^\infty |a_{\sigma(n)}|^p w_n \le \Vert (a_{\sigma(n)})_{n=1}^\infty\Vert_g^p \le K
\Vert  f \Vert_g^p.
\]
Taking the supremum in $\sigma$ we get $\Vert f\Vert_d\le \Vert f\Vert_g$. In other words,
$(\dd_n)_{n=1}^\infty\lesssim_K(\gb_n)_{n=1}^\infty$. The estimate
 $(\gb_n)_{n=1}^\infty\lesssim_1(\dd_n)_{n=1}^\infty$ is a consequence of Proposition~\ref{embeddings}.  

(v) $\Rightarrow$ (iii) Let $(\yy_n)_{n=1}^\infty$ be a basic sequence in $d(\ww,p)$ equivalent to the canonical basis of $g(\ww,p)$. By Corollary~\ref{garling-weakly-null} $(\yy_n)_{n=1}^\infty$ is weakly null. Then, via 
Theorem~\ref{BPSP} 
and passing to a subsequence, we obtain that  $(\yy_n)_{n=1}^\infty$ is equivalent to a (semi-normalized) block basic sequence of $(\dd_n)_{n=1}^\infty$.  By \cite[Corollary 2]{ACL73} and
passing to a further subsequence, we have 
 $(\dd_n)_{n=1}^\infty\lesssim (\yy_n)_{n=1}^\infty$. Since
 $
(\gb_n)_{n=1}^\infty\lesssim_1(\dd_n)_{n=1}^\infty
$
we are done.

(vi) $\Rightarrow$ (vii) Let $P$ be a projection from $g(\ww,p)$ onto $I_{d,g}(Y)$ and $S$ be the inverse map of
$I_{d,g}|_Y$. We have that $S\circ P\circ I_{d,g}$ preserves  a copy of $d(\ww,p)$.

(vii) $\Rightarrow$ (iii) By 
\cite[Theorem 5.5]{KPSTT12},  the sequence  $((T\circ I_{d,g})(\dd_n))_{n=1}^\infty$ does not converges to zero in norm. This means that we can find  a subsequence $(T(\gb_{n_k}))_{k=1}^\infty$ that is semi-normalized.  As $(\gb_n)_{n=1}^\infty$ is weakly null by Corollary~\ref{garling-weakly-null}, so is $(T(\gb_{n_k}))_{k=1}^\infty$.  Pass to a further subsequence if necessary so that, by applying 
Theorem~\ref{BPSP},
$(T(\gb_{n_k}))_{k=1}^\infty$ is (semi-normalized) and equivalent to a block basic sequence of $(\dd_n)_{n=1}^\infty$.  Again by \cite[Corollary 2]{ACL73} we may assume, passing to a further subsequence, that $(T(\gb_{n_k}))_{k=1}^\infty\gtrsim(\dd_n)_{n=1}^\infty$.  Hence,
\[(\dd_n)_{n=1}^\infty\lesssim(T(\gb_{n_k}))_{k=1}^\infty\lesssim(\gb_{n_k})_{k=1}^\infty\approx(\gb_n)_{n=1}^\infty\lesssim(\dd_n)_{n=1}^\infty.\]\end{proof}

\begin{remark}Notice that,  by Theorem~\ref{main}~(viii),  the inclussion map $I_{d,g}\colon  d(\ww,p) \to g(\ww,p)$  is never strictly singular. 
\end{remark}

\begin{remark}
Observe that $g(\ww,p)$ is never isometric in the natural way to $d(\ww,p)$.  Just consider the minimum $k\in\NN$ such that $w_{k+1}<1$, so that, letting
\[0<\alpha<\left(\frac{1-w_{k+1}}{k}\right)^{1/p} \]
and $f=\ee_{k+1} + \alpha \sum_{n=1}^k \ee_n$
we have $\Vert f \Vert_g=1<\Vert f \Vert_d$.
\end{remark}

Before constructing weights for which $(\gb_n)_{n=1}^\infty$ is not symmetric we need a few definitions.

\begin{definition}Let $\ww=(w_n)_{n=1}^\infty$ be a weight.
\begin{enumerate}
\item[(a)] $\ww$ is said to be \textbf{essentially decreasing} if $\sup_{k\le n} w_n/w_k<\infty$.

\item[(b)]   $\ww$  is said to be  \textbf{regular} if
 \[
\sup_m \frac{1}{m w_m} \sum_{n=1}^m w_n<\infty.
 \]
 \item[(c)] The \textbf{conjugate weight} $\ww^*=(w_n^*)_{n=1}^\infty$ is defined by 
\[
w^*_n=\frac{1}{nw_n}, \quad n\in\NN.
\]
\item[(d)]  $\ww$ is said to be \textbf{bi-regular} if both 
 $\ww$ and $\ww^*$ are regular weights.
 \end{enumerate}

\end{definition}

\begin{lemma}\label{SecondRegularLemma}  A weight  $\ww$ is essentially decreasing  if and only if there exists a normalized non-increasing weight $\vv$  such that
$\ww\approx \vv$.
\end{lemma}

\begin{proof} If  $(w_n)_{n=1}^\infty$ is   an essentially decreasing weight then the weight  $\vv=(v_n)_{n=1}^\infty$ given by 
$v_n=(\inf_{k\le n} w_k)/w_1$ does the trick. The converse is (also) trivial.  \end{proof}

\begin{lemma}\label{ThirdRegularLemma}  If a  weight $\ww$ is regular then  $\ww\notin \ell_1$ and $\ww^*\in c_0$. \end{lemma}

\begin{proof}   Assume, by contradiction, that $\ww=(w_n)_{n=1}^\infty \in \ell_1$, that is, $\sum_{n=1}^{\infty}w_n\approx 1$. Then, $1/m\lesssim w_m$ for  $m\in\NN$, and so $\ww\notin\ell_1$.
Now assume, again by contradiction, that $\ww^*=(w_n)_{n=1}^\infty \notin c_0$. Then there is an infinite subset $A\subseteq\NN$ with $n w_n\approx 1$ for $n\in A$. Consequently, 
$\sum_{n=1}^m w_n \lesssim 1$ for $m\in A$ and so $\ww\in\ell_1$.
\end{proof}

\begin{lemma}\label{FirstRegularLemma} Let  $\ww=(w_n)_{n=1}^\infty$ be an essentially decreasing regular weight. Then: 
\begin{enumerate}
\item[(i)]$mw_m \approx \sum_{n=1}^m w_n$ for $m\in\NN$.
\item[(ii)] $\ww^*$  also is essentially decreasing.
\end{enumerate}
\end{lemma}
\begin{proof} (i) Let $C=\sup_{k\le n} w_n/ w_k$.  For any $m\in\NN$ we have
 \[
 m w_m \le  C\sum_{n=1}^m w_n.
 \]
   
   (ii)  By  Part~(i), $\ww^{*}\approx (1/(\sum_{n=1}^m w_n))_{m=1}^\infty$.  We infer from Lemma~\ref{ThirdRegularLemma}  that $\ww^*$ is  essentially decreasing.
\end{proof}

\begin{lemma}\label{BiRegularResult} A bi-regular weight is essentially decreasing if and only if
$mw_m \lesssim \sum_{n=1}^m w_n$ for $m\in\NN$.
\end{lemma}

\begin{proof}The ``only if'' part is a consequence of Lemma~\ref{FirstRegularLemma}. Assume that 
$mw_m \approx \sum_{n=1}^m w_n$ for $m\in\NN$. By Lemma~\ref{ThirdRegularLemma}
$\ww^*$  is essentially decreasing. 
Then, by Lemma~\ref{FirstRegularLemma}, $\ww=\ww^{**}$ is essentially decreasing.
\end{proof}

 \begin{proposition}\label{SecondBiRegularResult}  Let $\ww=(w_n)_{n=1}^\infty$ be a weight. The following are equivalent:
 \begin{itemize}
 \item[(i)] $\ww$ is  essentially decreasing  and bi-regular;

  \item[(ii)] $\ww$ is    bi-regular and $\ww^*$ is essentially decreasing;
 
   \item[(iii)]  $\ww$ is essentially decreasing and
\begin{equation}\label{eq:2}
\sup_m \sum_{n=1}^m  \frac{w_{m+1-n}}{n w_n}<\infty.
\end{equation} 

  \item[(iv)]  $\ww^*$ is essentially decreasing  and  \eqref{eq:2} holds.
  \end{itemize}
\end{proposition}

\begin{proof} (i) $\Rightarrow$ (ii) is straightforward from Lemma~\ref{FirstRegularLemma} (ii),
 while  (ii) $\Rightarrow$ (i) is a consequence of 
  (i) $\Rightarrow$ (ii)  and the fact that $\ww^{**}=\ww$.
  
  (i) $\Rightarrow$ (iii) Let  $\ww^*=(w_n^*)_{n=1}^\infty$ and put 
\begin{align*}
C&=\sup_{k\ge n} \frac{w_k}{w_n}<\infty, \quad D=\sup_m w_m^* \sum_{n=1}^m w_n<\infty,\\
C^*&=\sup_{k\ge n} \frac{w_k^*}{w_n^*}<\infty, \text{ and } \quad D^*=\sup_m w_m \sum_{n=1}^m w_n^*<\infty. 
\end{align*}
Let $m\in\NN$ and choose $r\in\NN$ such that $2r-1\le m\le 2r$.
 We have
 \begin{align*}
 \sum_{n=1}^m  \frac{w_{m+1-n}}{n w_n} &=\sum_{n=1}^r w_{m+1-n} w_n^*+\sum_{n=r+1}^m w_{m+1-n} w_n^* \\
  &\le C w_{r} \sum_{n=1}^r  w_n^*+ C^* w_r^*\sum_{n=r+1}^m w_{m+1-n} \\
   &\le C w_{r} \sum_{n=1}^r  w_n^*+C^* w_r^*\sum_{n=1}^r w_n \\
    &=CD^*+C^* D.
  \end{align*}
as desired.

 (iii) $\Rightarrow$ (i)  Let $\ww^*=(w_n^*)_{n=1}^\infty$  and put
 \[
 C=\sup_{k\ge n} \frac{w_k}{w_n}<\infty  \text{ and } D=\sup_m w_m^* \sum_{n=1}^m w_n<\infty.
\]
 Let us prove that $\ww^*$  is   essentially decreasing.
 Let $k$ and $n$ be  integers with $1\le n\le k$. 
 Since the function $\eta$ given by $\eta(x):=\log(1+x)/x$ is decreasing on $(0,\infty)$,
 \begin{align*}
D&\ge  \sum_{j=1}^{n+k-1}   \frac{ w_{n+k-j}}{j w_j}
 \ge \sum_{j=k}^{n+k-1} \frac{ w_{n+k-j}}{jw_j}
\ge \frac{1}{C^2} \frac{ w_n}{w_k}  \sum_{j=k}^{n+k-1} \frac{1}{j}\\
&\ge \frac{1}{C^2} \frac{ w_n}{w_k}\int_k^{n+k} \frac{dx}{x} 
= \frac{1}{C^2}  \eta(n/k) \frac{w_k^*}{w_n^*} 
\ge \frac{1}{C^2}  \eta(1) \frac{w_k^*}{w_n^*}. 
\end{align*}
Hence  
\[
w_k^*\le (\log 2)^{-1} D  C^2   w_n^*, \quad n\le k.
\]
Now, for any $m\in\NN$,
\[
\frac{1}{m w_m} \sum_{n=1}^m w_n
=\frac{1}{m w_m} \sum_{n=1}^m w_{m+1-n}
\le  \frac{D C^2 }{\log 2}   \sum_{n=1}^m \frac{ w_{m+1-n}}{n w_n}\le \frac{D^2 C^2}{\log 2}\]
and
\[\frac{1}{m w_m^*} \sum_{n=1}^m w_n^*
=w_m  \sum_{n=1}^m \frac{1}{nw_n}
\le C  \sum_{n=1}^m \frac{ w_{m+1-n}}{n w_n}
\le DC.
\]
 
The equivalence (ii) $\Leftrightarrow$ (iv) is a consequence of  (i) $\Leftrightarrow$ (iii) and the simple  fact that
\[
  \sum_{n=1}^m \frac{ w_{m+1-n}}{n w_n}=  \sum_{n=1}^m \frac{ w_{m+1-n}^*}{n w_n^*},
 \]where  $\ww^*=(w_n^*)_{n=1}^\infty$. \end{proof}

\begin{theorem}\label{NonSymmetricResult}Let $\ww$ be a normalized non-increasing  bi-regular weight and $1\le p<\infty$. Then $\ww\in\WW$ and the canonical basis is not  a symmetric basis for $g(\ww,p)$.
\end{theorem}
\begin{proof}The  fact  that $\ww\in\WW$ is straightforward from Lemma~\ref{ThirdRegularLemma}.
For each $r\in \NN$, consider sequences $f^{(r)}=(a_{r,n})_{n=1}^\infty$ and
 $g^{(r)}=(b_{r,n})_{n=1}^\infty$ given by 
\[
\begin{cases}
a_{r,n} =b_{r,n}=0  & \text{ if } n>r, \\
a_{r,n} =b_{r,m+1-n}=(nw_n)^{-1/p}  & \text{ if } n\le r.
\end{cases}
\]
Notice that $g^{(r)}$ is a rearrangement of $f^{(r)}$. We have
\[
\sup_{r\in\NN} \Vert f^{(r)}\Vert_g
\ge \sup_{r\in\NN}  \Vert f^{(r)} \Vert_p
=\sup_{r\in\NN}\left(\sum_{n=1}^r \frac{1}{n}\right)^{1/p}=\infty.
\]
 By Lemma~\ref{FirstRegularLemma} (ii), $\ww^*=(w_n^*)_{n=1}^\infty$ is an essentially decreasing sequence.
 Let $r\in\NN$ and $\phi\in\OO$. Denote by $m(r,\phi)$ the largest integer $m$
such that $\phi(m)\le r$. We have   $\phi(n)\le n+r-m$ for $1\le n\le m$. Therefore,
\begin{align*}
\sup_{r\in\NN}\Vert  g^{(r)} \Vert^p_g&=\sup_{r\in\NN,\,\phi\in\OO}\sum_{n=1}^\infty |b_{r,\phi(n)}|^p w_n\\
& =\sup_{r\in\NN,\,\phi\in\OO} \sum_{n=1}^{m(r,\phi)}  w_n w_{r+1-\phi(n)}^*\\
&\lesssim\sup_{r\in\NN,\,\phi\in\OO}  \sum_{n=1}^{m(r,\phi)}  w_n w_{m(r,\phi)+1-n}^*\\
&= \sup_{m\in\NN}  \sum_{n=1}^{m}  w_n w_{m+1-n}^*\\
&=\sup_{m\in\NN} \sum_{n=1}^m \frac{w_{m+1-n}}{n w_n}.
\end{align*}
 Proposition~\ref{SecondBiRegularResult} yields that the canonical basis is not symmetric.
\end{proof}
 
 To give relief to Theorem~\ref{NonSymmetricResult} we need to exhibit examples of  normalized non-increasing bi-regular weights. To that end,
 in light of Lemma~\ref{ThirdRegularLemma}, if suffices to give examples of essentially decreasing bi-regular weights.
 Next, we give a very general procedure for constructing this kind of weights. For instance, we will prove that, for $0<a<1$ and $b\in\RR$,   $((\log(1+n))^b n^{-a})_{n=1}^\infty$  is a essentially decreasing bi-regular weight.
 
\begin{definition}A weight $(u_n)_{n=1}^\infty$ is said to be \textbf{asymptotically constant} if 
\[
\lim_j \sup_{2^{j-1} \le k, n \le 2^j} \frac{u_n} { u_k} =1.
\]
\end{definition}
 
 \begin{proposition}Let $0<a<1$ and $(u_n)_{n=1}^\infty$ be an asymptotically constant weight. Then $\ww=(n^{-a}u_n)_{n=1}^\infty$ is 
 an essentially decreasing bi-regular weight. \end{proposition}
 
 \begin{proof} Let $b=1-a$ and  consider $\vv=(v_n)_{n=1}^\infty$ given by  $v_n=1/u_n$. We have 
 $\ww^*=(n^{-b}v_n)_{n=1}^\infty$, with $0<b<1$ and $\vv$ asymptotically constant. Then, taking into account Lemma~\ref{BiRegularResult}
  it suffices to prove 
\[\sum_{n=1}^m n^{-a} u_n \approx t_m:=m^{1-a} u_m,\quad m\in\NN.\]
To that end, put
\[
c_j=\sup_{2^{j-1} \le n < 2^j} u_n, \quad j\in\NN,
\]
and let $(a_n)_{n=1}^\infty$ be  the unique weight such that $a_{2^{j-1}}=c_j$ for $j\in\NN$ and 
$(a_n)_{n=2^{j-1}}^{2^j}$ is in arithmetic progression. Define $(s_n)_{n=0}^\infty$ by $s_0=0$ and
$s_n=n^{1-a} a_n$ for $n\in\NN$. 
Since 
\[
\lim_j \frac{c_{j+1}}{c_j}=1,
\]
we have $t_n\approx s_n$ for $n\in\NN$. Thus we need only show that 
\[s_n-s_{n-1}\approx n^{-a} u_n, \quad n\in\NN.\] Given $n\ge 2$, pick $j=j(n)\in\NN$ such that
$2^{j-1}<n\le 2^j$. We have
\begin{align*}
\frac{s_n-s_{n-1}}{n^{-a} u_n}&=\frac{n(a_n-a_{n-1})}{u_n}+\frac{n^{1-a}-(n-1)^{1-a}}{n^{-a}}\frac{a_{n-1}}{ u_n}\\
&=\frac{n}{2^{j-1}}\frac{c_j}{u_n}+n\left(1-\left(1-\frac{1}{n}\right)^{1-a}\right)\frac{a_{n-1}}{ u_n}\\
&\approx 1.
\end{align*}
 \end{proof}
 
 We close with a most natural question that our work leaves open:
 
 \begin{question} Do there exist $\ww\in\WW$ and $1\leq p<\infty$ so that $g(\ww,p)$ admits a symmetric basis?
 \end{question}

\subsection*{Acknowledgment}
Ben Wallis thanks B\"unyamin Sar{\i} for his helpful comments.



\begin{bibsection}
\begin{biblist}

\bib{AK2016}{book}{
   author={Albiac, F.},
   author={Kalton, N.J.},
   title={Topics in Banach space theory},
   series={Graduate Texts in Mathematics},
   volume={233},
   edition={2},
   publisher={Springer, [Cham]},
   date={2016},
   pages={xx+508},
}

\bib{Al75}{article}{
   author={Altshuler, Z.},
   title={Uniform convexity in Lorentz sequence spaces},
   journal={Israel J. Math.},
   volume={20},
   date={1975},
   number={3-4},
   pages={260--274},
}

\bib{ACL73}{article}{
   author={Altshuler, Z.},
   author={Casazza, P. G.},
   author={Lin, Bor Luh},
   title={On symmetric basic sequences in Lorentz sequence spaces},
   journal={Israel J. Math.},
   volume={15},
   date={1973},
   pages={140--155},
}

\bib{CL74}{article}{
   author={Casazza, P. G.},
   author={Lin, B.L.},
   title={On symmetric basic sequences in Lorentz sequence spaces. II},
   journal={Israel J. Math.},
   volume={17},
   date={1974},
   pages={191--218},
}

\bib{CRS2007}{article}{
   author={Carro, M.J.},
   author={Raposo, J.A.},
   author={Soria, J.},
   title={Recent developments in the theory of Lorentz spaces and weighted
   inequalities},
   journal={Mem. Amer. Math. Soc.},
   volume={187},
   date={2007},
   number={877},
   pages={xii+128},
}

\bib{CJZ11}{article}{
   author={Chen, D.}
   author={Johnson, W.B.},
   author={Zheng, B.},
   title={Commutators on $(\sum\ell_q)_p$},
   journal={Studia Math.},
   volume={206},
   date={2011},
   number={2},
   pages={175--190},
}



 \bib{DKKT2003}{article}{
   author={Dilworth, S.J.},
   author={Kalton, N.J.},
   author={Kutzarova, D.},
   author={Temlyakov, V.N.},
   title={The thresholding greedy algorithm, greedy bases, and duality},
   journal={Constr. Approx.},
   volume={19},
   date={2003},
   number={4},
   pages={575--597},
}

\bib{DKOSZ2014}{article}{
   author={Dilworth, S. J.},
   author={Kutzarova, D.},
   author={Odell, E.},
   author={Schlumprecht, Th.},
   author={Zs{\'a}k, A.},
   title={Renorming spaces with greedy bases},
   journal={J. Approx. Theory},
   volume={188},
   date={2014},
   pages={39--56},
}


\bib{DOSZ11}{article}{
   author={Dilworth, S. J.},
   author={Odell, E.},
   author={Schlumprecht, Th.},
   author={Zs\'ak, A.},
   title={Renormings and symmetry properties of 1-greedy bases},
   journal={J. Approx. Theory},
   volume={163},
   date={2011},
   number={9},
   pages={1049--1075},
}


\bib{Ga68}{article}{
   author={Garling, D. J. H.},
   title={Symmetric bases of locally convex spaces},
   journal={Studia Math.},
   volume={30},
   date={1968},
   pages={163--181},
}




\bib{KPSTT12}{article}{
   author={Kami\'nska, A.},
   author={Popov, A.I.},
   author={Spinu, E.},
   author={Tcaciuc, A.},
   author={Troitsky, V.G.},
   title={Norm closed operator ideals in Lorentz sequence spaces},
   journal={J. Math. Anal. Appl.},
   volume={389},
   date={2012},
   number={1},
   pages={247--260},
}


\bib{LSZ14}{article}{
   author={Lin, P.},
   author={Sar{\i}, B.},
   author={Zheng, B.},
   title={Norm closed ideals in the algebra of bounded linear operators on Orlicz sequence spaces},
   journal={Proceedings of the American Mathematical Society},
   volume={142},
   date={2014},
   number={5},
   pages={1669--1680},
}

\bib{OS15}{article}{
   author={Oikhberg, T.},
   author={Spinu, E.},
   title={Subprojective Banach spaces},
   journal={J. Math. Anal. Appl.},
   volume={424},
   date={2015},
   number={1},
   pages={613--635},
}

\bib{Pu76}{article}{
   author={Pujara, L.R.},
   title={A uniformly convex Banach space with a Schauder basis which is
   subsymmetric but not symmetric},
   journal={Proc. Amer. Math. Soc.},
   volume={54},
   date={1976},
   pages={207--210},
}

\bib{Sa04}{article}{
   author={Sari, B.},
   title={Envelope functions and asymptotic structures in Banach spaces},
   journal={Studia Math.},
   volume={164},
   date={2004},
   number={3},
   pages={283--306},
}


\bib{Wh64}{article}{
   author={Whitley, R.J.},
   title={Strictly singular operators and their conjugates},
   journal={Trans. Amer. Math. Soc.},
   volume={113},
   date={1964},
   pages={252--261},
}


\bib{Zh14}{article}{
   author={Zheng, B.},
   title={Commutators on $(\sum\ell_q)_1$},
   journal={J. Math.  Anal.  Appl.},
   volume={413},
   date={2014},
   number={1},
   pages={284--290},
}



\end{biblist}
\end{bibsection}
\end{document}